\documentclass[a4paper,11pt,reqno]{amsart}
\usepackage[margin=1in]{geometry}
\usepackage{amsmath, amsthm, amssymb}
\usepackage[colorlinks=true, pdfstartview=FitV, linkcolor=blue,citecolor=blue, urlcolor=blue]{hyperref}
\usepackage[abbrev,lite,nobysame]{amsrefs}
\usepackage{times}
\usepackage[usenames,dvipsnames]{color}
\usepackage{graphicx}
\usepackage{subcaption}
\usepackage{mathtools}

\mathtoolsset{showonlyrefs=true}
\usepackage{dsfont}

\newtheorem{proposition}{Proposition}[section]
\newtheorem{theorem}[proposition]{Theorem}

\newtheorem{lemma}[proposition]{Lemma}
\newtheorem*{bootstrap*}{Bootstrap Step}
\theoremstyle{definition}

\newtheorem{remark}[proposition]{Remark}

\numberwithin{equation}{section}

\newcommand{\scalar}[2]{\langle #1, #2 \rangle}
\newcommand{\Scalar}[2]{\big\langle #1, #2 \big\rangle}

\newcommand{\dt}{\partial_t}
\newcommand{\T}{\mathbb{T}}
\newcommand{\ZZ}{\mathbb{Z}}
\newcommand{\R}{\mathbb{R}}
\newcommand{\N}{\mathbb{N}}
\newcommand{\dr}{\partial_r}
\newcommand{\drr}{\partial_{rr}}
\newcommand{\dtheta}{\partial_\theta}
\newcommand{\dthetatheta}{\partial_{\theta \theta}}
\newcommand{\Dt}{\frac{\dd}{\dd t}}

\newcommand{\norma}[2]{\lVert #1 \rVert_{#2}}
\newcommand{\Norma}[2]{\big\lVert #1 \big\rVert_{#2}}

\newcommand\uu {\boldsymbol{u}}
\def \l {\langle}
\def \r {\rangle}
\newcommand\dd{{\rm d}}
\newcommand\eps{\varepsilon}
\newcommand\e{{\rm e}}

\renewcommand{\aa}{\alpha} 
\newcommand{\bb}{\beta} 
\newcommand{\cc}{\gamma} 
\newcommand{\pp}{q} 

\title[Separation of time-scales in drift-diffusion equations]{Separation of time-scales in drift-diffusion equations on $\mathbb{R}^2$}

\author[M. Coti Zelati and M. Dolce]{Michele Coti Zelati and Michele Dolce}

\address{Department of Mathematics, Imperial College London, London, SW7 2AZ, UK}
\email{m.coti-zelati@imperial.ac.uk}

\address{GSSI - Gran Sasso Science Institute, 67100, L'Aquila, Italy}
\email{michele.dolce@gssi.it}



\begin{document}

	\begin{abstract}
		We deal with the problem of separation of time-scales and filamentation in a linear drift-diffusion problem posed 
		on the whole space $\R^2$. The passive scalar considered is stirred by an incompressible flow with radial symmetry.
		We identify a time-scale, much faster than the diffusive one, at which mixing happens \emph{along} the streamlines, 
		as a result of the interaction between transport and diffusion. This effect is also known as \emph{enhanced dissipation}.
		For power-law circular flows, this time-scale only depends on the behavior of the flow at the origin. The proofs are based 
		on an adaptation of a hypocoercivity scheme and yield a linear semigroup estimate in a suitable weighted $L^2$-based space.
	\end{abstract}
	\maketitle
\section{Introduction}

We consider the solution $f(t,x,y):[0,+\infty)\times \R^2 \to \R$ to the linear advection-diffusion equation 
\begin{equation}\label{eq:advecdiff}
	\begin{cases}
	\dt f+\uu\cdot \nabla f=\nu \Delta f, \quad &\text{in } \R^2, \ t\geq 0,\\
	f|_{t=0}=f^{in}, &\text{in }\R^2,
	\end{cases}
\end{equation}
where $\nu>0$ denotes the diffusivity coefficient, $f^{in}$ is an assigned mean-free initial datum, and $\uu:\R^2\to \R^2$ is a regular, radially symmetric and divergence-free 
velocity field given by
\begin{equation}\label{def:u}
	\uu(x,y)=\left(x^2+y^2\right)^{\pp/2}\begin{pmatrix}
	-y\\ x
	\end{pmatrix},
\end{equation}	
with $\pp\geq 1$ an arbitrary fixed exponent.
In this way, the background velocity field generates a counter-clockwise rotating motion, with a shearing effect across streamlines.
By passing to polar coordinates $(r,\theta)\in [0,\infty)\times \T$ in  \eqref{eq:advecdiff}, we deduce that
\begin{align}\label{eq:cauchycirc}
\begin{cases}
\dt f+ r^{\pp} \dtheta f=\nu\Delta f, \quad &\text{in } (r,\theta)\in [0,\infty)\times \T, \ t\geq 0,\\
f|_{t=0}=f^{in}, \quad &\text{in } (r,\theta)\in [0,\infty)\times \T.
\end{cases}
\end{align}
Here and in what follows, $\Delta$ denoted the Laplace operator in polar coordinates, namely
\begin{align}
\Delta=\drr+\frac{1}{r} \dr +\frac{1}{r^2}\dthetatheta.
\end{align}
The object of study of this paper is the sharp decay properties of the solution to \eqref{eq:cauchycirc}. Before stating our main result, 
we notice that the average of circles, defined for every $r\geq 0$ by
\begin{align}
\l f\r_\theta(t,r)=\frac{1}{2\pi}\int_\T f(t,r,\theta) \dd \theta,
\end{align}
satisfies the two-dimensional radially symmetric heat equation
\begin{align}\label{eq:thetaav}
\begin{cases}
\dt \l f\r_\theta=\nu \left(\drr+\frac{1}{r} \dr\right)\l f\r_\theta, \quad &\text{in } r\in [0,\infty), \ t\geq 0,\\
 \l f\r_\theta|_{t=0}= \l f^{in}\r_\theta, \quad &\text{in } r\in [0,\infty).
\end{cases}
\end{align}
Intuitively speaking, $\l f\r_\theta$ should have a decay time-scale proportional to $1/\nu$, in agreement with the classical diffusion
time-scale, while, due to the presence of the background flow, $f-\l f\r_\theta$ decays on a faster time-scale, depending on the 
exponent $\pp$. The main result of this paper is a precise quantification of this effect, in the $L^2$-based norm defined by
\begin{equation}\label{eq:Xnorm}
	\norma{g}{X}^2:=\int_0^\infty\int_\T |g(r,\theta)|^2 r\dd r\dd\theta +\int_0^\infty\int_\T r^{2(\pp-1)}|g(r,\theta)|^2 r\dd r\dd\theta.
\end{equation}
It can be stated as follows.
\begin{theorem}\label{th:mainth}
Let $\pp\geq 1$ and $f^{in}$ be such that $\|f^{in}\|_X<\infty$.
There exist constants $\eps_0\in(0,1)$ and $C_0\geq 1$ (explicitly computable and depending only on $\pp$) such that 
the following holds: for every $\nu\in(0,1]$
there hold the decay estimates 
\begin{align}\label{eq:heatdecay}
\|\l f(t)\r_\theta\|_{L^\infty}\leq \frac{C_0}{\sqrt{\nu t}}  \|\l f^{in}\r_\theta\|_{L^2}, \qquad \forall t\geq0,
\end{align}
and
\begin{align}\label{eq:L2est}
\norma{f(t)-\l f(t)\r_\theta}{X}\leq C_0 \e^{-\eps_0\lambda_{\nu} t}\norma{f^{in}-\l f^{in}\r_\theta}{X}, \qquad \forall t\geq0,
\end{align}
where 
\begin{equation}\label{eq:rate}
\lambda_{\nu}=\frac{\nu^{\frac{\pp}{\pp+2}}}{1+\frac{2(\pp-1)}{\pp+2}|\ln\nu|}
\end{equation}
is the decay rate.
\end{theorem}
	
Theorem \ref{th:mainth} is a quantification of the  \emph{shear-diffuse mechanism} that has been studied extensively in
the physics literature \cites{BajerEtAl01,DubrulleNazarenko94,RhinesYoung83,LatiniBernoff01}. Rigorous mathematical
results have started to appear only in recent times \cites{WEI18,Deng2013,BGM15III,CZDE18, CZEW19,CKRZ08,BCZGH15,WZ19,BCZ15,BMV14,BVW16,BGM15I,LiWeiZhang2017,BW13,Gallay2017,IMM17,WZZkolmo17,BGM15II,WZ18,LWZ18,GNRS18}, and the field has quickly attracted enormous interest. In general, the main difficulty is to \emph{quantify} the separation of time-scales that happens between
the evolution of $\theta$-independent mode and the others. The cause of this behavior is the interaction between transport and diffusion: roughly speaking, advection causes an energy cascade towards small spatial scales, at which dissipation kicks in and is more efficient.

In the radial case studied in this paper, the picture is quite intuitive and simple to describe: if $\nu\ll 1$, mixing  \emph{along} streamlines is most efficient 
on a time-scale between $1/\lambda_\nu$ and the classical diffusive one proportional to $1/\nu$. At these times, \eqref{eq:L2est} becomes arbitrarily small as $\nu\to0$ and the passive scalar tends to relax to 
a $\theta$-independent state, represented by the average $\l f\r_\theta$, which remains order 1 by \eqref{eq:heatdecay}. 
Then, diffusion takes over and dissipation happens mainly \emph{across} streamlines (see Figure \ref{fig:coffee3} below). 
From a quantitative standpoint, the enhanced dissipation mechanism is most efficient if the background flow is not ``too flat'' at the origin.
In fact, this is the same that happens in the case of shear flows, in which the decay rate is determined exclusively by the flatness of the critical
points (see \cite{BCZ15}). The presence of the log-correction in \eqref{eq:rate} is probably an artifact of the proof (notice that this contribution is no present
at $\pp=1$), but the decay rates are otherwise sharp (see also \cite{MD18} for numerical evidence).

\begin{figure}[h!]
  \centering
  \begin{subfigure}[b]{0.24\linewidth}
    \includegraphics[width=\linewidth]{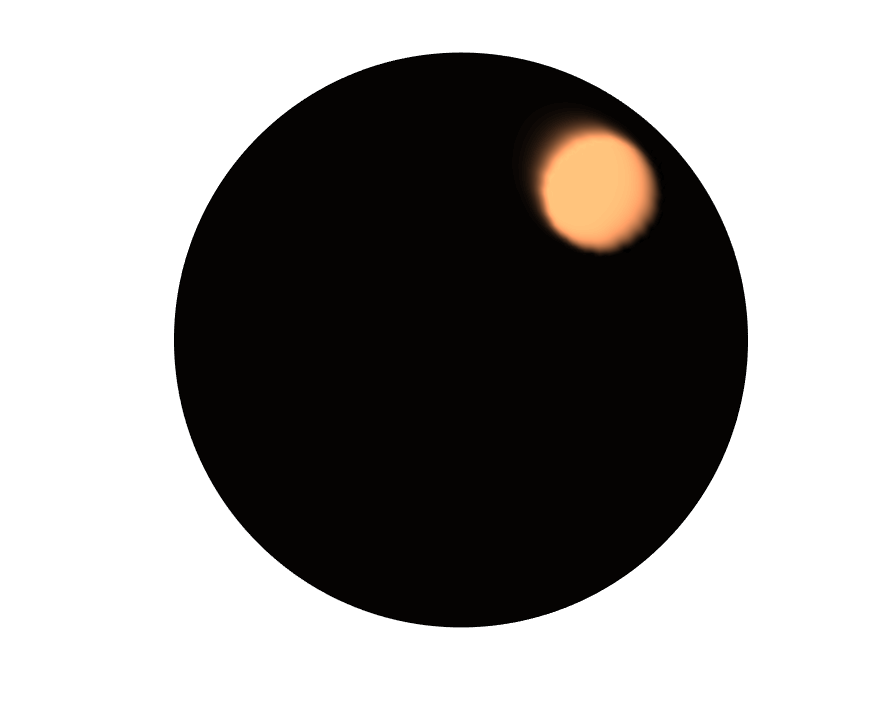}
  \end{subfigure}
  \begin{subfigure}[b]{0.24\linewidth}
    \includegraphics[width=\linewidth]{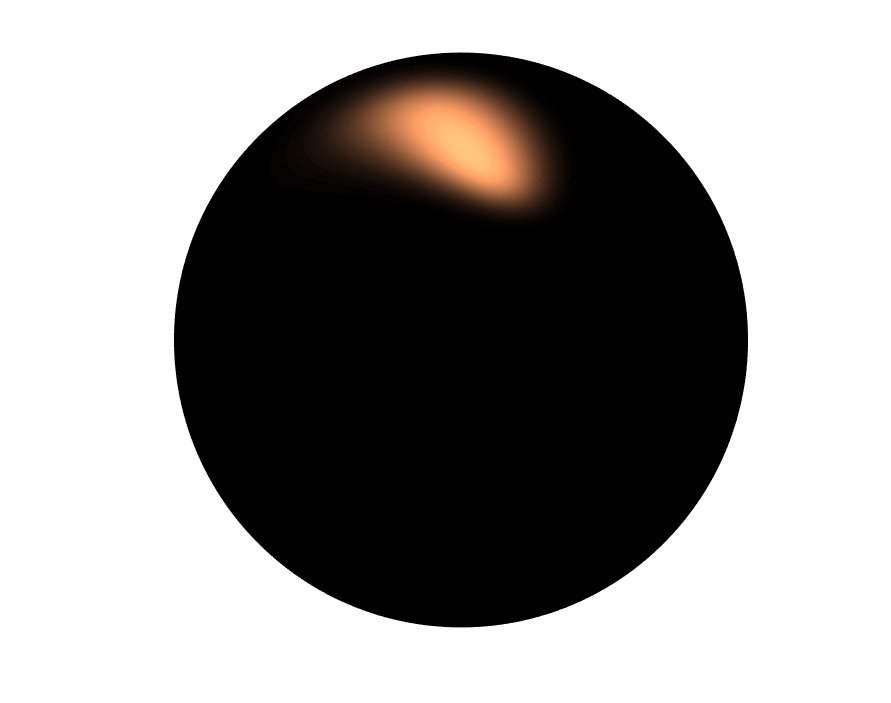}
  \end{subfigure}
  \begin{subfigure}[b]{0.24\linewidth}
    \includegraphics[width=\linewidth]{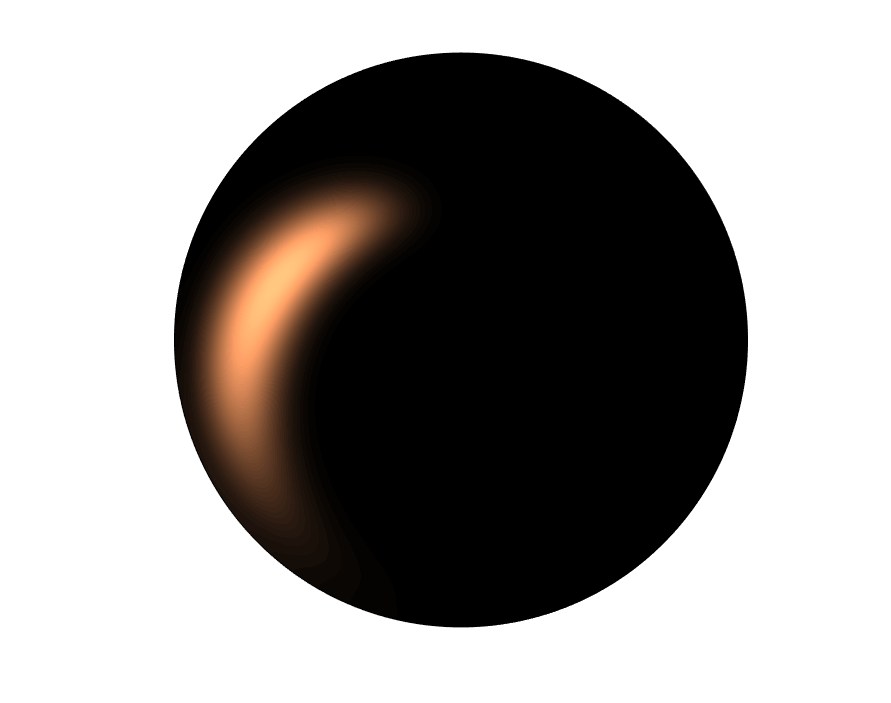}
  \end{subfigure}
  \begin{subfigure}[b]{0.24\linewidth}
    \includegraphics[width=\linewidth]{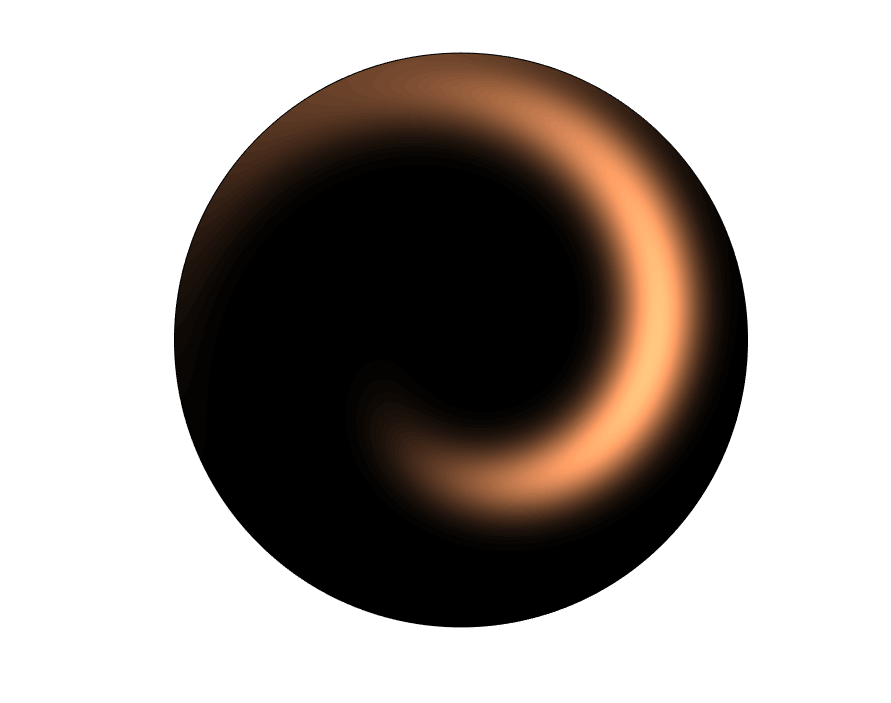}
  \end{subfigure}
    \begin{subfigure}[b]{0.24\linewidth}
    \includegraphics[width=\linewidth]{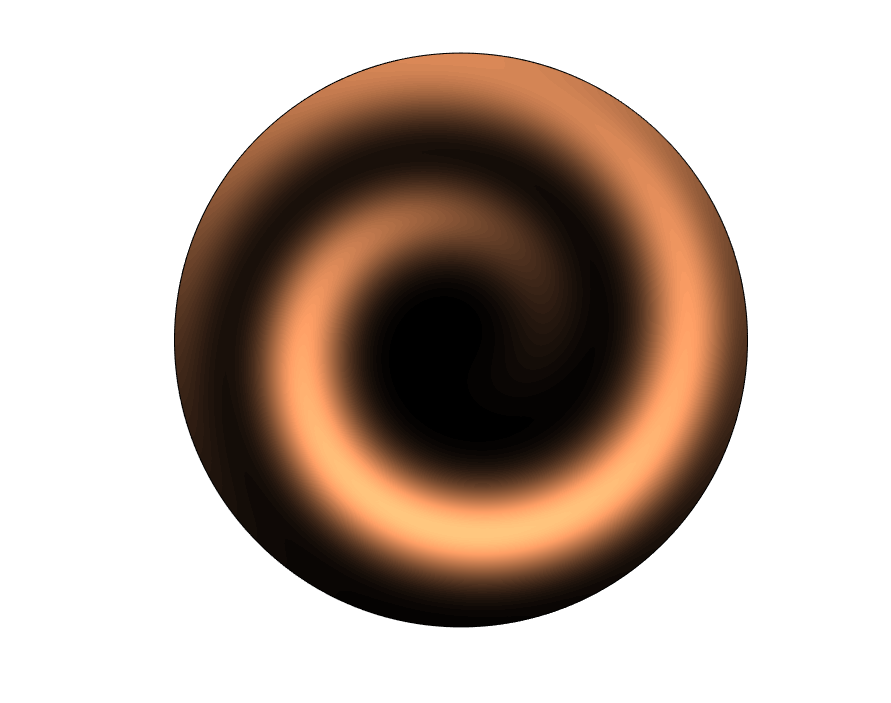}
  \end{subfigure}
  \begin{subfigure}[b]{0.24\linewidth}
    \includegraphics[width=\linewidth]{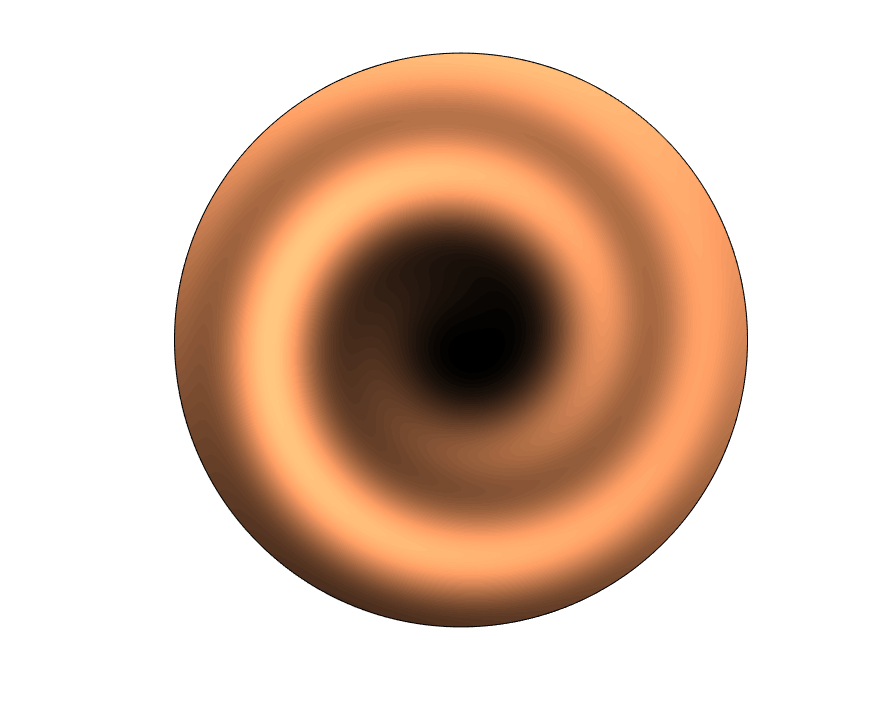}
  \end{subfigure}
  \begin{subfigure}[b]{0.24\linewidth}
    \includegraphics[width=\linewidth]{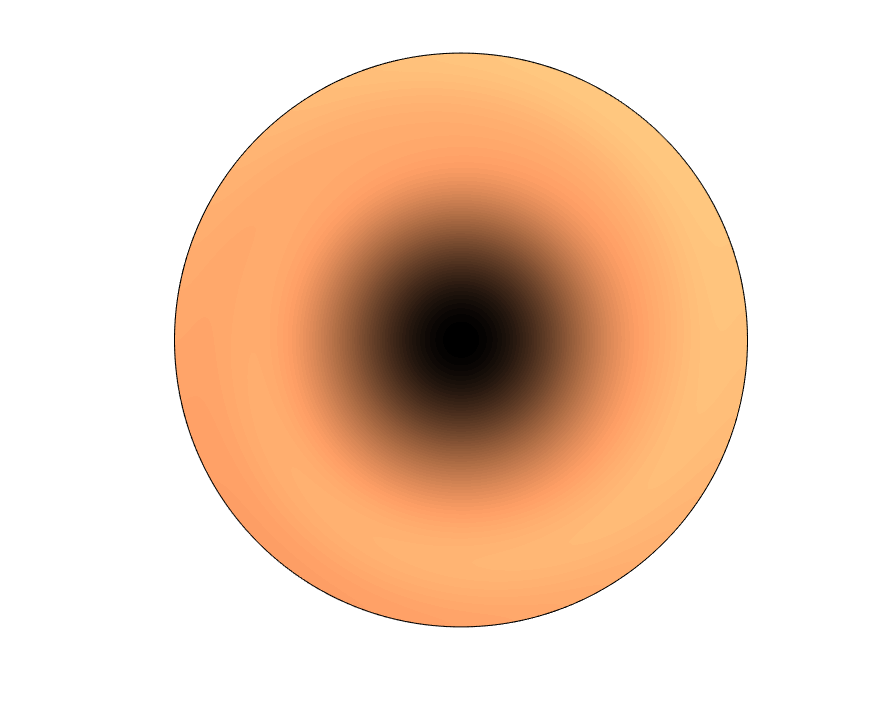}
  \end{subfigure}
  \begin{subfigure}[b]{0.24\linewidth}
    \includegraphics[width=\linewidth]{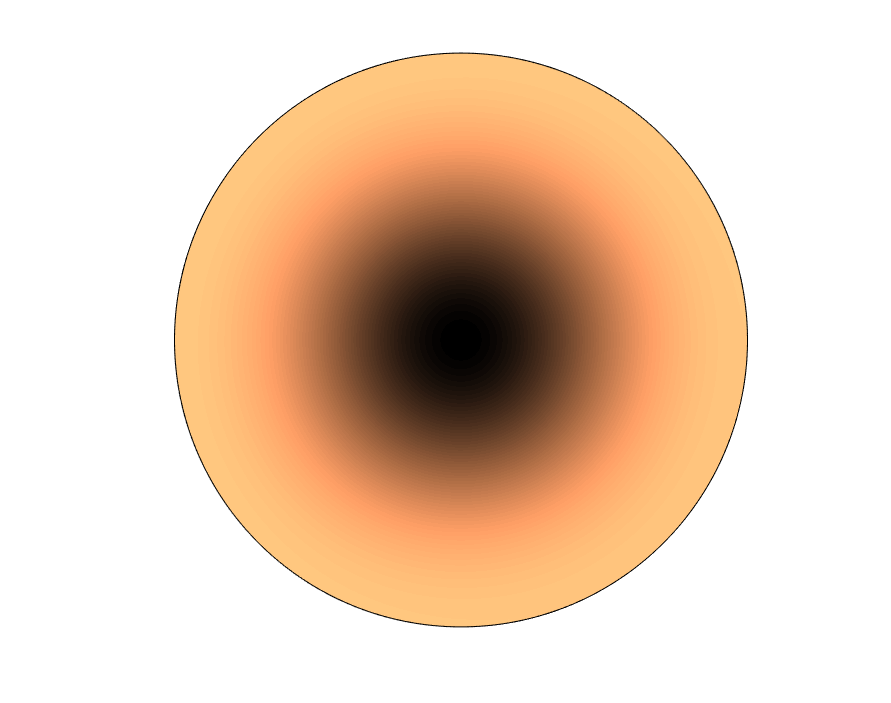}
  \end{subfigure}
  \caption{The evolution of an initial concentration subject radial stirring and small diffusion (think of milk stirred in a coffee mug!). At the beginning,
  advection is the main effect, while at the end the solution has reached radial symmetry and only undergoes diffusion.}
  \label{fig:coffee3}
\end{figure}

The present paper is the first one of its kind to treat the question of enhanced dissipation in the whole space $\R^2$. There are two main 
difficulties here: on the one hand, we do not have a Poincar\'e inequality, so that exponential decay estimates are, generically speaking, 
far from being trivial. On the other hand, the possibility of growth of the solution at infinity forces us to work in weighted spaces (hence
the $X$-norm defined in \eqref{eq:Xnorm}), adding a few technicalities in closing the estimates. 
The proof of Theorem \ref{th:mainth} relies on ideas originated in kinetic theory to study the long-time behavior of collisional models
\cites{DMS15,DesvillettesVillani01, HerauNier2004,Herau2007}, and is based on a technique known as hypocoercivity \cite{Villani09}. In
fluid dynamics problems, significant additional care is required to treat the infinite P\'eclet number limit $\nu\to 0$ (see \cites{CLS17,CZDE18}
for the mixing properties at $\nu=0$). 

Combining the ideas of this article and the those of \cite{BCZ15}, it is possible to treat the case of more general radial flows,
where $r^\pp$ in \eqref{eq:cauchycirc} is replaced by an arbitrary smooth function $u(r)$, and the case of a bounded domain (a disk), by imposing
suitable no-flux boundary conditions on $f$. In this case, the weights in the $X$-norm in \eqref{eq:Xnorm} become redundant, but no
substantial change in the rate \eqref{eq:rate} is expected, except possibly in the case $\pp=1$ (see \cite{BCZ15}).

\subsection{Outline of the article}
The next Section \ref{sec:enbal} is devoted to the derivation of several energy balances that will be crucial in order to set up the hypocoercivity 
scheme in Section \ref{sec:hypo}: this is where a Fourier-localized version of Theorem \ref{th:mainth} is proven. The proof is based
on two fundamental lemmas, that hold in general for radial functions, that are included at the end of the article in the Appendix \ref{app:ineq}.

\subsection{Notations and conventions}
In what follows, we will use $\|\cdot\|$ and $\l\cdot,\cdot\r$ for the standard real $L^2$ norm and scalar product, respectively
defined as
\begin{align}
\l g,\tilde{g}\r =\int_0^\infty\int_\T g(r,\theta) \tilde{g}(r,\theta) r\dd r\dd\theta, \qquad \|g\|^2=\int_0^\infty\int_\T |g(r,\theta)|^2 r\dd r\dd\theta.
\end{align}
Given $g\in L^2$, we can expand it in Fourier series in the angular $\theta$ variable, namely
\begin{align}
g(r,\theta)=\sum_{\ell\in \ZZ} \mathfrak{g}_\ell(r)\e^{i\ell\theta}, \qquad \mathfrak{g}_\ell(r)=\frac{1}{2\pi}\int_0^{2\pi}g(r,\theta)\e^{-i\ell\theta}\dd \theta.
\end{align} 
We will often make use of functions that are localized on a single band in $\theta$-frequency. Thus, for $k\in\N_0$ we set
\begin{equation}\label{eq:band}
 g_k(r,\theta):=\sum_{|\ell|=k} \mathfrak{g}_\ell(r)\e^{i\ell \theta}.
\end{equation}
This way we may write 
\begin{align}
g(r,\theta)=\sum_{k\in\N_0}g_k(r,\theta),
\end{align} 
as a sum of \emph{real-valued} functions $g_k$ that are localized in $\theta$-frequency on a single band $\pm k$, $k\in\N_0$. In particular,
for the $\theta$-average of a function $g$, we have $\l g\r_\theta=g_0=\mathfrak{g}_0$.

We will not distinguish between the two dimensional $L^2(r\dd r\dd\theta)$  and the one dimensional $L^2(r\dd r)$ spaces, as no dimensional property will be used. Notice that for $\theta$-independent functions, the norms only differ by a constant.

\subsection{Differential operators} In polar coordinates, the gradient and Laplace operators become
\begin{align}
\nabla=\begin{pmatrix}
	\dr \\ \frac{1}{r}\dtheta
	\end{pmatrix}, \qquad \Delta=\drr+\frac{1}{r} \dr +\frac{1}{r^2}\dthetatheta,
	\end{align}
respectively. Whenever we apply them to a function localized in one frequency band $k\in\N_0$ (see \eqref{eq:band}), 
we will often tacitly use that
\begin{align}\label{eqnormink}
\|\nabla g_k\|^2= \|\dr g_k\|^2+ k^2 \left\|\frac{g_k}{r}\right\|^2.
\end{align}

\section{Energy balances}\label{sec:enbal}
In this section, we derive several energy balances for \eqref{eq:cauchycirc}, that will prove useful in Section \ref{sec:hypo} to set up
the hypocoercivity scheme needed to prove Theorem \ref{th:mainth}.

\begin{proposition}\label{propen}
Let $f$ be a smooth solution of \eqref{eq:cauchycirc}. Then there hold the energy balances
	\begin{align}
	\label{dtf}
\frac{1}{2}&\Dt \norma{f}{}^2+\nu \norma{\nabla f}{}^2=0,\\
\label{dtnablaf}\frac{1}{2}&\Dt \Norma{\nabla f}{}^2+\nu\Norma{\Delta f}{}^2=-\pp\Scalar{r^{\pp-1}\dtheta f}{\dr f},\\
\label{dtscalar} 
&\Dt \Scalar{r^{\pp-1}\dtheta f}{\dr f}+\pp\Norma{r^{\pp-1}\dtheta f}{}^2=-2\nu \Scalar{r^{\pp-1}\dr \dtheta f}{\Delta f}-\nu\pp\Scalar{r^{\pp-2}\dtheta f}{\Delta f},\\
\label{rdtheta}\frac{1}{2}&\Dt \Norma{r^{\pp-1} \dtheta f}{}^2+\nu \norma{r^{\pp-1}\dtheta \nabla f}{}^2=2\nu (\pp-1)^2\Norma{r^{\pp-2}\dtheta f}{}^2.
\end{align}
\end{proposition}

\begin{remark}
Notice that the most singular term in the origin is $\norma{r^{\pp-2}f_k}{}$ for $\pp \in(1,2)$. In view of \eqref{eq:fourexp}, if a $f$ is smooth we have that
\begin{align}
f_k(r,\theta)\sim r^k,\qquad \text{as } r\to 0,
\end{align}
making the term finite.
\end{remark}
\begin{proof}[Proof of Proposition \ref{propen}]
In the proof, we will use several times the antisymmetry of the the transport operator, namely that
\begin{align}\label{eq:anti}
\l r^\pp \dtheta f,g\r=-\l f,r^\pp \dtheta g\r,
\end{align}
for every $f,g$ sufficiently regular.

The equality \eqref{dtf} follows by multiplying \eqref{eq:cauchycirc} by $f$, integrating by parts and using \eqref{eq:anti}. To prove \eqref{dtnablaf}, multiply \eqref{eq:cauchycirc} by $r\Delta f $ and integrate to get
\begin{align}
	\Dt \scalar{f}{\Delta f}+\scalar{r^\pp \dtheta f}{\Delta f}= \nu \norma{\Delta f}{}^2,
\end{align} 
so that integrating by parts we have 
\begin{align}
	\frac{1}{2}\Dt\norma{\nabla f}{}^2+\Scalar{\nabla \big(r^\pp \dtheta f\big)}{\nabla f}+\nu\norma{\Delta f}{}^2=0.
\end{align}
	Then just observe that by \eqref{eq:anti} we have
	\begin{align}
	\Scalar{\nabla \big(r^\pp \dtheta f\big)}{\nabla f}=\Scalar{r^\pp \dtheta \nabla f}{\nabla f}+\pp \scalar{r^{\pp- 1}\dtheta f}{\dr f}=\pp \scalar{r^{\pp- 1}\dtheta f}{\dr f},
	\end{align}
so that \eqref{dtnablaf} follows.

We now turn to \eqref{dtscalar}. Using \eqref{eq:cauchycirc} we get
\begin{align}
\Dt \Scalar{\dtheta f}{r^{\pp-1}\dr f}=&\nu \left[ \Scalar{\dtheta \Delta f}{r^{\pp-1}\dr f}+\Scalar{\dtheta f}{r^{\pp-1}\dr \Delta f}\right]\notag\\
&-\Scalar{r^{\pp}\dthetatheta f}{r^{\pp-1}\dr f}-\Scalar{\dtheta f}{r^{\pp-1}r^{\pp}\dr\dtheta f}-\pp \Scalar{r^{\pp-1}\dtheta f}{r^{\pp-1}\dtheta f}.
\end{align}
Integrating by parts the third term in the right-hand side above, we have
\begin{align}
-\Scalar{r^{\pp}\dthetatheta f}{r^{\pp-1}\dr f}=\Scalar{r^\pp\dtheta f}{r^{\pp-1}\dtheta\dr f},
\end{align}
which cancels with the fourth term. Moreover,
\begin{align}
\Scalar{\dtheta f}{r^{\pp-1}\dr \Delta f}&=\Scalar{r^\pp\dtheta f}{\frac{1}{r}\dr \Delta f}=-\Scalar{\frac{1}{r}\dr \big(r^\pp \dtheta f\big)}{\Delta f}\notag\\
&=-\Scalar{r^{\pp-1}\dr \dtheta f}{\Delta f}-\pp \Scalar{r^{\pp-2}\dtheta f}{\Delta f}.
\end{align}
So thanks to these computations we get that 
\begin{align}
\Dt \Scalar{\dtheta f}{r^{\pp-1}\dr f}+\pp\norma{r^{\pp-1}\dtheta f}{}^2=-2\nu \Scalar{r^{\pp-1}\dr \dtheta f}{\Delta f}-\nu\pp\Scalar{r^{\pp-2}\dtheta f}{\Delta f},
\end{align}
hence proving \eqref{dtscalar}.

Finally, to prove \eqref{rdtheta}, we use \eqref{eq:anti} and compute the following 
\begin{align}
\frac{1}{2}\Dt \Norma{r^{\pp-1}\dtheta f}{}^2&=\nu \Scalar{r^{\pp-1}\dtheta f}{r^{\pp-1}\dtheta \Delta f}-\Scalar{r^{\pp-1}\dtheta f}{r^{\pp-1}r^{\pp}\dthetatheta f}\notag\\
&=\nu \Scalar{r^{2(\pp-1)}\dtheta f}{\dtheta \Delta f}=-\nu\Scalar{\nabla\big(r^{2(\pp-1)}\dtheta f\big)}{\dtheta \nabla f}.
\end{align}
Then computing the last term and using once more \eqref{eq:anti},  we have
\begin{align}
-\nu\Scalar{\nabla\big(r^{2(\pp-1)} \dtheta f\big)}{\dtheta \nabla f}=-\nu\Norma{r^{\pp -1}\dtheta \nabla f}{}^2-2\nu(\pp-1)\Scalar{r^{2(\pp-1)}\dtheta f}{\frac{1}{r}\dr \dtheta f}.
\end{align}
 Rewrite the last scalar product of the previous equality as 
\begin{align}
\Scalar{r^{2(\pp-1)}\dtheta f}{\frac{1}{r}\dr \dtheta f}&=-\Scalar{\frac{1}{r}\dr \big(r^{2(\pp-1)}\dtheta f\big)}{\dtheta f}\notag\\
&=-\Scalar{\frac{1}{r}\dr \dtheta f}{r^{2(\pp-1)}\dtheta f}-2(\pp-1)\Scalar{\frac{1}{r}r^{2(\pp-1)-1}\dtheta f}{\dtheta f},
\end{align}
or equivalently we have that 
\begin{align}
\Scalar{r^{2(\pp-1)}\dtheta f}{\frac{1}{r}\dr \dtheta f}=-(\pp-1)\Norma{r^{\pp-2}\dtheta f}{}^2.
\end{align}
So putting everything together we get 
\begin{align}
\frac{1}{2}\Dt \Norma{r^{\pp-1} \dtheta f}{}^2=-\nu \norma{r^{\pp-1}\dtheta \nabla f}{}^2+2\nu (\pp-1)^2\Norma{r^{\pp-2}\dtheta f}{}^2,
\end{align}
proving \eqref{rdtheta}. The proof is over.
\end{proof}

\section{Hypocoercivity setting}\label{sec:hypo}
To prove Theorem \ref{th:mainth}, we proceed in several steps. Firstly, we decouple equation \eqref{eq:cauchycirc} in the various 
Fourier modes in $\theta$. Then we set up a proper energy functional that satisfies a suitable inequality that allows to prove an
enhanced decay at the rate \eqref{eq:rate} \emph{without} the log-correction. As this functional involves norms of higher derivatives, 
we then show how to pass to the $L^2$ estimate \eqref{eq:L2est} by giving up a log-correction in the rate. All the estimates that we perform in this section are done in a smooth setting. Then one recover the general case by a standard approximation argument and passing to the limit.

\subsection{The equation mode-by-mode}
By taking the Fourier transform in $\theta$ of \eqref{eq:cauchycirc}, we deduce that for each $\ell\in \ZZ$ we have
\begin{align}\label{eq:cauchyfourier}
\begin{cases}
\dt \mathfrak{f}_\ell+ i\ell r^{\pp} \mathfrak{f}_\ell =\nu\left( \drr +\frac1r \dr - \frac{\ell^2}{r^2}\right) \mathfrak{f}_\ell, \quad &\text{in } r\in [0,\infty), \ t\geq 0,\\
\mathfrak{f}_\ell |_{t=0}=\mathfrak{f}_\ell^{in}, \quad &\text{in } r\in [0,\infty).
\end{cases}
\end{align}
It is apparent that the equation decouples in the Fourier modes, and one can treat each equation separately. However, this implies a great deal
of notation due to the fact that the various $\mathfrak{f}_\ell$'s are complex functions. Instead, we prefer to study \eqref{eq:cauchycirc} with
an initial datum concentrated on a Fourier band $|\ell|=k\in \N_0$, since none of the estimates involved are the same for the mode $k$ and $-k$.
Hence, we will study, as explained in \eqref{eq:band}, the evolution of
\begin{equation}
f_k(t,r,\theta):=\sum_{|\ell|=k} \mathfrak{f}_\ell(t,r)\e^{i\ell \theta},
\end{equation}
which satisfies  \eqref{eq:cauchycirc} but also has nice property with respect to norms, such as \eqref{eqnormink}.

\begin{remark}[The zeroth mode]
As mentioned in the introduction, $f_0$ satisfies the diffusion equation \eqref{eq:thetaav}, which is precisely the two-dimensional heat equation
under radial symmetry assumptions. Estimate \eqref{eq:heatdecay} is therefore a classical one, which can be proven via the Green's function
of the heat operator (see e.g. \cite{GGS10}). 
\end{remark}

\subsection{The modified energy functional}
Throughout the section, we assume that $k\geq1$. 
The main result is the following theorem.

\begin{theorem}\label{th:enhanced}
Let $\pp\geq 1$.
There exists a constant $\eps_0\in (0,1)$  such that the following holds:
there exist positive numbers $\aa_0,\bb_0,\cc_0$ only depending on $\eps_0$
such that for each integer $k\geq 1$ and $\nu>0$ with
$\nu k^{-1}\leq 1$
 the energy functional
\begin{align}\label{eq:PHIk}
\Phi_k=\frac12\left[\norma{f_k}{}^2+\aa_0 \frac{\nu^{\frac{2}{\pp+2}}}{k^{\frac{2}{\pp+2}}}\norma{\nabla f_k}{}^2
+2\pp \bb_0 \frac{\nu^\frac{2-\pp}{\pp+2}}{k^\frac{4}{\pp+2}}\scalar{r^{\pp-1}\dtheta f_k}{\dr f_k}
+\cc_0 \frac{\nu^{\frac{-2(\pp-1)}{\pp+2}}}{k^{\frac{6}{\pp+2}}}\norma{r^{\pp-1}\dtheta f_k}{}^2 \right]
\end{align}
satisfies the differential inequality
\begin{align}\label{eq:diffPHI}
\Dt \Phi_k+2\varepsilon_0\nu^{\frac{\pp}{\pp+2}}k^{\frac{2}{\pp+2}}\Phi_k\leq 0,\end{align}
for all $t\geq 0$.
\end{theorem}

The main Theorem \ref{th:mainth} follows as a consequence of Theorem \ref{th:enhanced}. In particular, we prove Theorem \ref{th:mainth} in Section \ref{sec:recnorm}.

In order to prove Theorem \ref{th:enhanced}, we need the two crucial Lemmas \ref{lemma:spectralgap} and  \ref{lemma:ralpha-2}, which we state
in greater generality in the Appendix \ref{app:ineq}. We now proceed
with the proof of the main theorem of this section.
\begin{proof}[Proof of Theorem \ref{th:enhanced}]
	For simplicity of notation we will always omit the subscript $k$ throughout all the proof. Also, in the course of this proof, $c_1, c_2$ and $c_3$ will denote specific constants, depending on $\pp$ and independent of $\nu,k, \alpha_0,\beta_0,\gamma_0$. Define the following energy functional
\begin{equation}
\Phi=\frac{1}{2}\left[\norma{f}{}^2+\aa\norma{\nabla f}{}^2+2\pp \bb\scalar{r^{\pp-1}\dtheta f}{\dr f}+\cc\norma{r^{\pp-1}\dtheta f}{}^2\right],
\end{equation}
where $\aa,\bb,\cc$ will be chosen later, complying with several constraints that will appear during the proof. First of all, observe that 
\begin{equation*}
2\pp \bb|\scalar{r^{\pp-1}\dtheta f }{\dr f}|\leq 2\pp \bb \norma{r^{\pp-1}\dtheta f}{}\norma{ \dr f}{}\leq \frac{\aa}{2}\norma{\nabla f}{}^2+\frac{2\pp^2\bb^2}{\aa}\norma{r^{\pp-1}\dtheta f}{}^2.
\end{equation*}
As a first condition on the coefficients $\aa, \bb, \cc$, we require that
\begin{equation}
\label{b2/ac}
\frac{\bb^2}{\aa \cc}\leq \frac{1}{4\pp^2}.
\end{equation}
In this way, \eqref{b2/ac} guarantees that 
\begin{equation}
\label{equivalencephi}
\frac{1}{4}\left(2\norma{f}{}^2+\aa\norma{\nabla f}{}^2+\cc\norma{r^{\pp-1}\dtheta f}{}^2\right)\leq \Phi \leq \frac{1}{4}\left(2\norma{f}{}^2+3\aa\norma{\nabla f}{}^2+3\cc\norma{r^{\pp-1}\dtheta f}{}^2\right).
\end{equation}
Then, thanks to the energy balances given in Proposition \ref{propen}, we have that 
\begin{equation}
\label{eq:dtphi}\begin{split}
\Dt \Phi + &\nu \norma{\nabla f}{}^2+\aa \nu \norma{\Delta f}{}^2+\bb\pp^2 \norma{r^{\pp-1} \dtheta f}{}^2+\cc\nu \norma{r^{\pp-1}\dtheta \nabla f}{}^2\\
=&-\aa\pp\Scalar{r^{\pp-1}\dtheta f}{\dr f}-2\bb \pp \nu \Scalar{r^{\pp-1}\dr \dtheta f}{\Delta f}\\
&-\bb\pp^2\nu \Scalar{r^{\pp-2}\dtheta f}{\Delta f}+2\cc\nu(\pp-1)^2\norma{r^{\pp-2}\dtheta f}{}^2.
\end{split}
\end{equation}
Now we need to estimate the terms on the right-hand side of \eqref{eq:dtphi}. We control the scalar product terms just by Cauchy-Schwartz inequality. In particular, we have that
\begin{equation}
\label{bd:ascalar}
\aa\pp \Scalar{r^{\pp-1}\dtheta f}{\dr f}\leq \frac{\nu}{4}\norma{\nabla f}{}^2+\frac{\aa^2\pp^2}{\nu}\norma{r^{\pp-1}\dtheta f}{}^2,
\end{equation}
and 
\begin{equation}
\label{bd:bscalar1}
2\bb\pp\nu \Scalar{r^{\pp-1}\dr \dtheta f}{\Delta f} \leq \frac{\aa \nu}{4}\norma{\Delta f}{}^2+\frac{4\bb^2\pp^2\nu}{\aa}\norma{r^{\pp-1}\dtheta \nabla f}{}^2.
\end{equation}
Moreover,
\begin{equation}
\label{bd:bscalar2}
\begin{split}
\bb \pp^2\nu \Scalar{r^{\pp-2}\dtheta f}{\Delta f}&\leq \frac{\aa \nu}{4}\norma{\Delta f}{}^2+\frac{\bb^2 \pp^4 \nu}{\aa}\norma{r^{\pp-2}\dtheta f}{}^2\\
&\leq \frac{\aa \nu}{4}\norma{\Delta f}{}^2+\frac{\bb^2 \pp^4 \nu}{\aa}\norma{r^{\pp-1}\nabla f}{}^2,
\end{split} 
\end{equation}
where the last inequality  follows since $\nabla=(\dr, \dtheta/r)$. Now assume that 
\begin{equation}
\label{restrictionab}
\frac{\aa^2}{\bb}\leq \frac{\nu}{4},
\end{equation}
in order to absorb the last term of \eqref{bd:ascalar} on the l.h.s. of \eqref{eq:dtphi}.
So thanks to \eqref{bd:ascalar}, \eqref{bd:bscalar1} and \eqref{bd:bscalar2}, we obtain that 
\begin{equation}
 \label{bd:dtphirest}\begin{split}
 \Dt \Phi +&\frac{3\nu}{4}\norma{\nabla f}{}^2+\frac{\aa\nu}{2}\norma{\Delta f}{}^2+\frac{3\bb\pp^2}{4}\Norma{r^{\pp-1}\dtheta f}{}^2+\nu\cc\Norma{r^{\pp-1} \dtheta \nabla f}{}^2\\
 \leq&2\cc\nu (\pp-1)^2\norma{r^{\pp-2}\dtheta f}{}^2+\frac{4\bb^2\pp^2\nu}{\aa}\norma{r^{\pp-1}\dtheta \nabla f}{}^2+ \frac{\bb^2 \pp^4 \nu}{\aa}\norma{r^{\pp-1}\nabla f}{}^2\\
=&2\cc\nu (\pp-1)^2\norma{r^{\pp-2}\dtheta f}{}^2+\frac{2\bb^2\pp^2\nu}{\aa}\bigg(2+\frac{q^2}{2k^2}\bigg)\norma{r^{\pp-1}\dtheta \nabla f}{}^2,
 \end{split}
\end{equation}
where the last one follows since we are localized at frequencies $k$.
Now further restrict \eqref{b2/ac} as follows:
\begin{equation}
\label{restrictionabc}
\frac{\bb^2}{\aa\cc}\leq \frac{1}{4q^2(2+q^2/2)}=:\frac{1}{c_1}.
\end{equation}
Thanks to \eqref{restrictionabc} and the fact that $k\geq 1$, we can absorb the last term on the right-hand side of \eqref{bd:dtphirest} in the left-hand side to infer that 
\begin{equation}
\label{bd:dtphiproof1}
\Dt \Phi +\frac{3\nu}{4}\norma{\nabla f}{}^2+\frac{\aa \nu}{2}\norma{\Delta f}{}^2+\frac{3\bb\pp^2}{4}\Norma{r^{\pp-1} \dtheta f}{}^2+\frac{\cc\nu}{2}\Norma{r^{\pp-1}\dtheta \nabla f}{}^2\leq 2\cc\nu (\pp-1)^2\norma{r^{\pp-2}\dtheta f}{}^2.
\end{equation}
It remains to estimate the last term above, which is not present for $\pp=1$ (see Remark \ref{rem:pp=1} below). This will be done by Lemma \ref{lemma:ralpha-2}, by choosing
\begin{equation}\label{eq:sigmachoice}
\sigma=\frac{\nu}{2\bb\pp^2 k^2}.
\end{equation}
Since $k\geq 1$, we find that there exists some constant $c_2\geq 1$ such that
\begin{align}\label{bd:ralpha-2}
\frac{1}{c_2} \left(\frac{\nu}{k^2}\right)^{\frac{1}{\pp}}\bb^{\frac{\pp-1}{\pp}}\norma{r^{\pp-2}\dtheta f}{}^2&\leq \frac{\nu}{8}\norma{\nabla f}{}^2+\frac{\bb\pp^2}{4}\norma{r^{\pp-1}\dtheta f}{}^2,
\end{align}
where, for $c_\pp$ as in Lemma \ref{lemma:ralpha-2}, we can define 
\begin{align}
\frac{1}{c_2}:= \frac{1}{c_\pp}
\pp^{\frac{2(\pp-1)}{\pp}}2^{-(2+\frac{1}{\pp})}.
\end{align}
Let us impose for the moment that
\begin{equation}
\label{restrictionbc}
2\cc \nu (\pp-1)^2\leq \frac{1}{c_2} \left(\frac{\nu}{k^2}\right)^{\frac{1}{\pp}}\bb^{\frac{\pp-1}{\pp}}.
\end{equation}
In this way, from \eqref{bd:dtphiproof1} and \eqref{bd:ralpha-2} we obtain
\begin{equation}
\label{bd:dtphiproof2}
\Dt \Phi +\frac{5\nu}{8}\norma{\nabla f}{}^2+\frac{\bb\pp^2}{2}\Norma{r^{\pp-1} \dtheta f}{}^2\leq 0.
\end{equation}
Now, in order to satisfy the constraints \eqref{restrictionab}, \eqref{restrictionabc} and \eqref{restrictionbc} for every $\nu,k$, we rescale
$\aa,\bb,\cc$ in such a way so that the inequalities become independent of $\nu,k$. That is, we choose $\aa$, $\bb$ and $\cc$
as in \eqref{eq:PHIk}, namely,
\begin{equation}
\label{chooseabc}
\aa=\aa_0 \frac{\nu^{\frac{2}{\pp+2}}}{k^{\frac{2}{\pp+2}}}, \qquad \bb=\bb_0 \frac{\nu^\frac{2-\pp}{\pp+2}}{k^\frac{4}{\pp+2}}, \qquad \cc=\cc_0 \frac{\nu^{-\frac{2(\pp-1)}{\pp+2}}}{k^{\frac{6}{\pp+2}}},
\end{equation}
for some $\aa_0, \bb_0, \cc_0$ to be chosen properly later but \emph{independent} of $\nu$ and $k$. To see why exactly those exponents, let us look for example at the scaling in $\nu$. Assume that $\aa=\nu^i$, $\bb=\nu^j$ and $\cc=\nu^n$. Then by \eqref{restrictionab}, \eqref{restrictionabc} and \eqref{restrictionbc}, one needs to impose that 
\begin{align}
2i=1+j,\qquad 2j=n+i, \qquad 1+n=\frac{1}{\pp}+j\bigg(\frac{\pp-1}{\pp}\bigg).
\end{align}
Then solving explicitly the system we infer \eqref{chooseabc}. One argues analogously also for the scaling in $k$. All that is missing now
is the norm $f$ in \eqref{bd:dtphiproof2}. With this in mind, we apply  Lemma \ref{lemma:spectralgap} with the same choice of $\sigma$ as
in \eqref{eq:sigmachoice} and deduce that
\begin{align}
\label{bd:normal2}
\frac{1}{c_3} \nu^{\frac{\pp-1}{\pp}}\left(k^2\bb\right)^{\frac{1}{\pp}}\norma{f}{}^2\leq \frac{\nu}{8}\norma{\nabla f}{}^2+\frac{\bb\pp^2}{4}\norma{r^{\pp-1}\dtheta f}{}^2, \qquad c_3:= 2^\frac{3\pp-1}{\pp}\pp^{-\frac{2}{\pp}}.
\end{align}
Using this in \eqref{bd:dtphiproof2}, we find
\begin{equation}
\Dt \Phi +\frac{1}{c_3} \nu^{\frac{\pp-1}{\pp}}\left(k^2\bb\right)^{\frac{1}{\pp}}\norma{f}{}^2+\frac{\nu}{2}\norma{\nabla f}{}^2+\frac{\bb\pp^2}{4}\Norma{r^{\pp-1} \dtheta f}{}^2\leq 0.
\end{equation}
In view of \eqref{chooseabc}, this becomes
\begin{equation}\label{bd:dtphifactout}
\Dt \Phi+ \frac{1}{c_3} \nu^{\frac{\pp}{\pp+2}}k^{\frac{2}{\pp+2}}\bb_0^{\frac{1}{\pp}}\norma{f}{}^2+\frac{\nu}{2}\norma{\nabla f}{}^2+\frac{\bb\pp^2}{4}\Norma{r^{\pp-1} \dtheta f}{}^2\leq 0,
\end{equation}
or, equivalently,
\begin{equation}
\Dt \Phi+ \frac{1}{c_3} \nu^{\frac{\pp}{\pp+2}}k^{\frac{2}{\pp+2}}\bb_0^{\frac{1}{\pp}}\left[2\norma{f}{}^2
+\frac{c_3}{\aa_0\bb_0^{\frac{1}{\pp}}}\aa\norma{\nabla f}{}^2
+c_3\frac{\bb_0^\frac{\pp-1}{\pp}}{ \cc_0}\cc\Norma{r^{\pp-1} \dtheta f}{}^2\right]\leq 0.
\end{equation}
Here it is crucial that in front of $\aa$ and $\cc$ in the brackets, we have something independent of $\nu$ and $k$. In fact we want to use equivalence \eqref{equivalencephi}, so assuming 
\begin{align}\label{restb0a0}
\frac{c_3}{ \aa_0\bb_0^{\frac{1}{\pp}}}\geq 3,
\end{align}
and
\begin{align}\label{restb0c0}
c_3\frac{\bb_0^\frac{\pp-1}{\pp}}{ \cc_0}\geq 3,
\end{align}
we conclude that
\begin{equation}\label{eq:diffineq}
\Dt \Phi +2\varepsilon_0\nu^{\frac{\pp}{\pp+2}}k^{\frac{2}{\pp+2}} \Phi \leq 0,
\end{equation}
where $\varepsilon_0$ is determined as follows: let
\begin{align}
\delta=\min\left\{\frac{1}{2c_2 (\pp-1)^2},\frac{c_3}{3}\right\},
\end{align}
and choose $\aa_0, \bb_0, \cc_0$ as 
\begin{equation}
\aa_0=\frac{c_1}{\delta}\bb_0^{\frac{q+1}{q}}, \qquad \bb_0=\min\left\{\frac{\delta^2}{4c_1^2}, \frac{c_3\delta}{3c_1}\right\}^\frac{\pp}{\pp+2}, \qquad \cc_0=\delta \bb_0^\frac{\pp-1}{\pp},
\end{equation}
and it is straightforward to verify that $\aa_0,\bb_0,\cc_0$ satisfies \eqref{restrictionab}, \eqref{restrictionabc}, \eqref{restrictionbc}, \eqref{restb0a0} and \eqref{restb0c0}.
Then $\varepsilon_0$ could be chosen as
\begin{equation}
\varepsilon_0= \frac{\beta_0^\frac{1}{\pp}}{2c_3} ,
\end{equation}
hence proving \eqref{eq:diffineq}. The proof is concluded.

\end{proof}

\begin{remark}[The case $\pp=1$]\label{rem:pp=1}
As noted in the proof, the case $\pp=1$ does not require estimating the most dangerous error term in \eqref{bd:dtphiproof1}. In fact, 
one does not even need to include the $\cc$-term in the functional, since from \eqref{eq:PHIk} it is apparent that the $\cc$-term is 
of the same order as $\| f_k\|$. This is also the reason why Lemma \ref{lemma:ralpha-2} is not needed in this case.
\end{remark}

\subsection{Reconstruction of the $X$-norm}
\label{sec:recnorm}
Theorem \ref{th:enhanced} does not directly imply the decay in the $X$-norm as stated in Theorem \ref{th:mainth}. In this section we prove Theorem \ref{th:mainth} in the case $\pp\geq 1$, and for initial data localized at band $k\in\N$. It is convenient to define the functional
\begin{equation}
W_k(t)=\frac{1}{2}\norma{f_k(t)}{}^2+\frac{\cc_0}{4}\norma{r^{\pp-1}f_k(t)}{}^2,
\end{equation}
which is equivalent to the $X$-norm. Here we prove the following localized version of Theorem \ref{th:mainth}.
\begin{proposition}\label{pr:Wkappa}
Let $\pp\geq 1$ and $\nu k^{-1} \in (0,1]$. Then
\begin{equation}\label{eq:Wkap}
W_k(t)\leq C_0W_k(0) \exp\left(-\frac{2\varepsilon_0\nu^{\frac{\pp}{\pp+2}}k^{\frac{2}{\pp+2}}}{1+\frac{2(\pp-1)}{\pp+2}\left(|\ln\nu|+\ln k\right)}t\right),
\end{equation}
for all $t\geq 0$, and where $C_0\geq 1$ is a constant independent of $\nu,k$.
\end{proposition}

Notice that Proposition \ref{pr:Wkappa} implies a linear semigroup estimate for \eqref{eq:cauchyfourier} in a weighted $L^2$-space. 
Theorem \ref{th:mainth}, and in particular \eqref{eq:L2est}, simply follows by summing over $k\in \N_0$ in \eqref{eq:Wkap}.

\begin{proof}[Proof of Proposition \ref{pr:Wkappa}]
First of all, recall the definition of $\gamma$ given in \eqref{chooseabc}. Then observe that for $\nu k^{-1} \in (0,1]$ we get 
\begin{equation}
\label{bd:gammak2}
\frac{\cc}{\cc_0} k^2=\left(\frac{k}{\nu}\right)^{\frac{2(\pp-1)}{\pp+2}}\geq 1,
\end{equation}
since $\pp\geq 1$. So, in view of the equivalence \eqref{equivalencephi} for the functional $\Phi_k$, we infer that 
\begin{equation}
\label{bd:wphi}
W_k(t)\leq \frac12\norma{f_k(t)}{}^2+\frac{\gamma}{4}\norma{r^{\pp-1}\dtheta f_k(t)}{}^2\leq \Phi_k(t). 
\end{equation}
Now we introduce the following two fixed times
\begin{equation}
\label{def:Tnuk}
T_{\nu,k}:= \frac{1}{2\varepsilon_0\nu^{\frac{\pp}{\pp+2}}k^{\frac{2}{\pp+2}}}, \qquad T_{\nu,k,\ln}:= \frac{1+\frac{2(\pp-1)}{\pp+2}\left(|\ln \nu|+\ln k\right)}{2\varepsilon_0\nu^{\frac{\pp}{\pp+2}}k^{\frac{2}{\pp+2}}},
\end{equation}
where $T_{\nu,k}$ is for technical convenience and $T_{\nu,k,\ln}$ is the expected time scale for the enhanced dissipation mechanism. We divide
the proof in two cases.

\subsection*{Case 1}
For $t\in (0,T_{\nu,k,\ln})$, Theorem \ref{th:mainth} is proved if we are able to show that 
\begin{equation}
\label{bd:enhanshort}
W_k(t)\leq C_1W_k(0).
\end{equation}
Notice that for $\pp=1$, this is trivial. In all other cases,
by \eqref{dtf}, we have that
\begin{equation}
\label{energyeq}
\norma{f_k(T)}{}^2+2\nu \int_{0}^T\norma{\nabla f_k(\tau)}{}^2\dd\tau =\norma{f_k^{in}}{}^2, \qquad \forall T\geq 0.
\end{equation} 
In particular, we know that $\norma{f_k(\cdot)}{}^2$ is monotonically decreasing, and
\begin{equation}
\label{bd:decrf}
\norma{f_k(t)}{}^2\leq \norma{f_k^{in}}{}^2. 
\end{equation}
Now we need to infer some property on $\norma{r^{\pp-1}f_k(t)}{}^2$, in particular we cannot hope in general for some monotonicity. Then we proceed as follows: consider the equality \eqref{rdtheta} and define $C_\pp=4(\pp-1)^2$. By Lemma \ref{lemma:ralpha-2} we have that 
\begin{align}
\Dt \norma{r^{\pp-1}f_k}{}^2\leq \nu C_\pp^\pp\norma{\nabla f_k}{}^2+\nu \norma{r^{\pp-1}f_k}{}^2.
\end{align}
Then for any $t\in(0,T_{\nu,k,\ln})$ we infer that
\begin{equation}\label{eq:buffon}
\begin{split}
\norma{r^{\pp-1} f_k(t)}{}^2&\leq \left( \norma{r^{\pp-1} f_k^{in}}{}^2+\nu C_\pp^\pp \int_{0}^{t}\e^{-\nu s}\norma{\nabla f_k(s)}{}^2 \dd s\right)\e^{\nu t}\\
&\leq \left(\norma{r^{\pp-1} f_k^{in}}{}^2+\frac{C_\pp^\pp}{2} \norma{f_k^{in}}{}^2\right)\e^{\nu t},
\end{split}
\end{equation}
where in the last line we have bounded $\e^{-\nu s}$ by $1$ and used the energy equality \eqref{energyeq}.
Finally notice that 
\begin{equation}
\nu T_{\nu,k,\ln}= \left(\frac{\nu}{k}\right)^{2/(\pp+2)}\frac{1+\frac{2(\pp-1)}{\pp+2}\left(|\ln \nu|+\ln k\right)}{2\varepsilon_0}\leq \widetilde{C}_1,
\end{equation} 
for some $\widetilde{C}_1$ which depends only on $q$ and $\varepsilon_0$. The last inequality follows simply because for any $\beta>0$ we have $\lim_{x\to 0}x^\beta \ln x=0$ and we are assuming that $\nu k^{-1} \in (0,1]$. So from \eqref{eq:buffon} we get that 
\begin{equation}
\label{bd:ralpha-1t*}
\norma{r^{\pp-1} f_k(t)}{}^2\leq \left(\norma{r^{\pp-1} f_k^{in}}{}^2+\frac{C_\pp^\pp}{2} \norma{f_k^{in}}{}^2\right)\e^{\widetilde{C}_1}.
\end{equation} 
By combining \eqref{bd:decrf} with \eqref{bd:ralpha-1t*} we prove \eqref{bd:enhanshort} for a proper $C_1$, hence proving Proposition  \ref{pr:Wkappa} for all $t\leq T_{\nu,k,\ln}$.
\end{proof}

\subsection*{Case 2}
Now we pass to the case $t\geq T_{\nu,k,\ln}$. First of all, consider the energy equality  \eqref{energyeq} for $T=T_{\nu,k}$.
Invoking the mean value theorem, there exists
\begin{equation}
\label{def:t*} t^*\in (0,T_{\nu,k}),
\end{equation} 
such that
\begin{equation*}
2\nu T_{\nu,k}\norma{\nabla f_k(t^*)}{}^2\leq \norma{f_k^{in}}{}^2,
\end{equation*}
but recalling the definition of $\aa$, see \eqref{chooseabc}, we infer that
\begin{equation}
\label{bd:a/a0f}
\frac{\aa}{\aa_0}\norma{\nabla f_k(t^*)}{}^2\leq \varepsilon_0\norma{f_k^{in}}{}^2.
\end{equation}
Using \eqref{bd:a/a0f} in the equivalence \eqref{equivalencephi} for $\Phi_k$, we get that 
\begin{equation}
\label{bd:phit*}
\Phi_k(t^*)\leq \frac{1}{2}\norma{f_k(t^*)}{}^2+3\aa_0\varepsilon_0 \norma{f_k^{in}}{}^2+\frac{3}{4}\cc k^2\norma{r^{\pp-1}f_k(t^*)}{}^2.
\end{equation}
Since $t^*\leq T_{\nu,k}\leq T_{\nu,k,\ln}$, by combining \eqref{bd:decrf} and \eqref{bd:ralpha-1t*} with \eqref{bd:phit*}, we conclude that 
\begin{equation}
\label{bd:phit*W0}
\begin{split}
\Phi_k(t^*)&\leq 3\frac{\cc k^2}{\cc_0}\left[\left(\frac{\cc_0}{6\cc k^2}+\frac{\aa_0\varepsilon_0\gamma_0}{\gamma k^2}\right)\norma{f_k^{in}}{}^2+\frac{\gamma_0}{4}\norma{r^{\pp-1}f_k^{in}}{}^2\right] \\
&\leq \widetilde{C}_2\gamma k^2W_k(0),
\end{split}
\end{equation}
where the last one follows by \eqref{bd:gammak2} and $\widetilde{C}_2$ is a constant that does not depend on $\nu,k$.
Then by \eqref{bd:wphi}, \eqref{bd:phit*W0} and applying the differential inequality \eqref{eq:diffPHI} starting from $t^*$, we have that 
\begin{align}
W_k(t)&\leq \Phi_k(t)\leq \e^{-2\varepsilon_0\nu^{\frac{\pp}{\pp+2}}k^{\frac{2}{\pp+2}}(t-t^*)}\Phi_k(t^*)\notag\\
&\leq \widetilde{C}_2 \e^{2\varepsilon_0\nu^{\frac{\pp}{\pp+2}}k^{\frac{2}{\pp+2}}t^*}\cc k^2\e^{-2\varepsilon_0\nu^{\frac{\pp}{\pp+2}}k^{\frac{2}{\pp+2}}t}W_k(0)\notag\\
&\leq \widetilde{C}_2\gamma_0 \e\frac{\cc k^2}{\gamma_0}\e^{-2\varepsilon_0\nu^{\frac{\pp}{\pp+2}}k^{\frac{2}{\pp+2}}t}W_k(0),
\end{align}
where in the last inequality we have used the fact that $t^*\leq T_{\nu,k}$. For $\pp=1$, we are done, since $\cc k^2\sim1$. For $\pp>1$, notice that
\begin{align}
\e\frac{\cc k^2}{\gamma_0}\e^{-2\varepsilon_0\nu^{\frac{\pp}{\pp+2}}k^{\frac{2}{\pp+2}}t}=\e\left(\frac{k}{\nu}\right)^{\frac{2(\pp-1)}{\pp+2}}\e^{-2\varepsilon_0\nu^{\frac{\pp}{\pp+2}}k^{\frac{2}{\pp+2}}t}= \e^{1+\frac{2(\pp-1)}{\pp+2}\left(|\ln \nu|+\ln k\right)-2\varepsilon_0\nu^{\frac{\pp}{\pp+2}}k^{\frac{2}{\pp+2}}t}.
\end{align}
Now let $a=1+\frac{2(\pp-1)}{\pp+2}\left(|\ln\nu|+\ln k\right)$ and $b=2\varepsilon_0\nu^{\frac{\pp}{\pp+2}}k^{\frac{2}{\pp+2}}$. By the definition of $T_{\nu,k,\ln}$, we are interested in times $t\geq a/b$. Then observe the following basic inequality 
\begin{align}
a-bt\leq 1-\frac{b}{a}t, \qquad \text{for } t\geq \frac{a}{b},
\end{align}
which is true for any $a>1$ and $b>0$. 
So finally we have that 
\begin{equation}
W_k(t)\leq C_2 W_k(0)\exp\left(-\frac{2\varepsilon_0\nu^{\frac{\pp}{\pp+2}}k^{\frac{2}{\pp+2}}}{1+\frac{2(\pp-1)}{\pp+2}\left(|\ln\nu|+\ln k\right)}t\right),
\end{equation}
where $C_2=\e\widetilde{C}_2\gamma_0$. Then by choosing $C_0=\max \{C_1,C_2\}$ we conclude the proof of Proposition  \ref{pr:Wkappa}.

\begin{remark}[No log-correction for $\pp=1$]
We stress once more that for $\pp=1$, there is no logarithmic correction for $W_k(t)$ in \eqref{eq:Wkap}. This is essentially due to the fact
that the $\cc$-term is of the same order as $\|f_k\|$.
\end{remark}

\appendix

\section{Radial Fourier expansion and two useful inequalities}\label{app:ineq}
We prove here two useful inequalities for functions localized to a Fourier band as in \eqref{eq:band}. Namely, given $g\in L^2$,
we write
\begin{align}
g(r,\theta)=\sum_{k\in\N_0}g_k(r,\theta),
\end{align} 
with
\begin{equation}\label{eq:band2}
 g_k(r,\theta):=\sum_{|\ell|=k} \mathfrak{g}_\ell(r)\e^{i\ell \theta}, \qquad \mathfrak{g}_\ell(r)=\frac{1}{2\pi}\int_0^{2\pi}g(r,\theta)\e^{-i\ell\theta}\dd \theta.
\end{equation}
It is worth noticing here that if a function $g$ is smooth, then 
\begin{align}\label{eq:fourexp}
 \mathfrak{g}_\ell(r)\sim r^{|\ell|}(a_0+a_1 r^2+a_2 r^4+\ldots),
\end{align}
as it can be seen by simply applying the operator $\Delta$ to $ \mathfrak{g}_\ell$ (see \cite{BG98}). This is useful when performing integration 
by parts and make sure that no term at $r=0$ arises.

The first inequality is reminiscent of the spectral gap in \cite[Proposition 2.7]{BCZ15}, originally inspired by \cite{GNN09}.

\begin{lemma}
	\label{lemma:spectralgap}
	Let $\pp\geq 1$, and let $g\in H^1$ such that $r^{\pp-1}g\in L^2$. Assume $k\in\N$, and let $g_k$ be defined as in \eqref{eq:band2}. 
	For any $\sigma >0$, it holds that 
	\begin{equation}
	\label{bd:spectral}
	\sigma^{\frac{\pp -1}{\pp}} \left\|g_k\right\|^2\leq\sigma\left\|\frac{g_k}{r}\right\|^2+\norma{r^{\pp-1}g_k}{}^2\leq \sigma \norma{\nabla g_k}{}^2+\norma{r^{\pp-1}g_k}{}^2.
	\end{equation}
\end{lemma}

\begin{proof}[Proof of Lemma \ref{lemma:spectralgap}] 
The proof will be performed for a smooth function $g$. It is clear that the original statement holds upon passing to the limit in the various norms.
Let $R>0$ to be chosen, then 
	\begin{align}
	\sigma^{\frac{\pp-1}{\pp}}\norma{g_k}{}^2&=\sigma^{\frac{\pp-1}{\pp}}\iint_{r\leq R} r^2 \frac{|g_k|^2}{r^2}r\dd r \dd\theta+\sigma^{\frac{\pp-1}{\pp}}\iint_{r>R} \frac{1}{r^{2(\pp-1)}}r^{2(\pp-1)}|g_k|^2 r\dd r\dd\theta\\
	&\leq \sigma^{\frac{\pp-1}{\pp}}R^{2}\left\|\frac{g_k}{r}\right\|^2+\sigma^{\frac{\pp-1}{\pp}}R^{-2(\pp-1)}\norma{r^{\pp-1}g_k}{}^2,
	\end{align}
	where in the last line we use the fact that $q\geq 1$. Hence choosing 
	\begin{equation}
	R=\sigma^{\frac{1}{2\pp}},
	\end{equation}
	we infer that 
	\begin{equation}
	\sigma^{\frac{\pp -1}{\pp}} \left\|g_k\right\|^2\leq\sigma\left\|\frac{g_k}{r}\right\|^2+\norma{r^{\pp-1}g_k}{}^2.
	\end{equation}
	The second inequality simply follows from the fact that  
	$\norma{\nabla g_k}{}^2=\norma{\dr g_k}{}^2+k^2\norma{g_k/r}{}^2$. Hence we have proved the lemma.
\end{proof}
The second lemma enters in the main proofs as a way to essentially control the behavior at the origin of each $g_k$. 
It is stated as follows.

\begin{lemma}\label{lemma:ralpha-2}
Let $\pp\geq 1$, and let $g\in H^1$ such that $r^{\pp-1}g\in L^2$. Assume $k\in\N$, and let $g_k$ be defined as in \eqref{eq:band2}. 
	There exists a constant $c_\pp\geq 2$ depending
	only on $\pp$ such that
	\begin{equation}
	\label{bd:ralpha2}
	\frac{1}{c_\pp}\sigma^{\frac{1}{\pp}}\norma{r^{\pp-2}g_k}{}^2\leq \sigma \norma{\nabla g_k}{}^2+\norma{r^{\pp-1} g_k}{}^2,
	\end{equation}
	for any $\sigma >0$.
\end{lemma}
Lemma \ref{lemma:ralpha-2} is crucial in order to close the estimate in the weighted space $X$ that appears in the main Theorem \ref{th:mainth}.

\begin{remark}
%
	The exponents of $\sigma$ given in Lemma \ref{lemma:spectralgap} and \ref{lemma:ralpha-2} are the natural ones given by a simple scaling argument.
\end{remark}

\begin{proof}[Proof of Lemma \ref{lemma:ralpha-2}]
We preliminary note that the case $\pp=1$ is trivial, since $\| g_k/r\|\leq \|\nabla g_k\|$ so \eqref{bd:ralpha2} holds with $c_\pp=1$. As for the rest, we divide the proof in two cases, namely $\pp\in (1,2)$ and $\pp\geq 2$.
	\subsection*{Case $\pp\geq 2$} In this case, let $R>0$ to be chosen later. Then we get that 
	\begin{align}
	\sigma^{\frac{1}{\pp}}\norma{r^{\pp-2}g_k}{}^2&=\sigma^{\frac{1}{\pp}}\iint_{r\leq R} r^{2(\pp-2)}|g_k|^2 r\dd r\dd\theta+\sigma^{\frac{1}{\pp}} \iint_{r>R}\frac{1}{r^2} r^{2(\pp-1)}|g_k|^2 r\dd r\dd\theta \notag\\
	&\leq \sigma^{\frac{1}{\pp}}R^{2(\pp-2)}\norma{g_k}{}^2+\sigma^{\frac{1}{\pp}}R^{-2}\norma{r^{\pp-1}g_k}{}^2,
	\end{align}
	where in the last line, for the first term, we have used the fact that $\pp\geq 2$. Now we take advantage of Lemma \ref{lemma:spectralgap}, so choose 
	\begin{equation}
	R=\sigma^{\frac{1}{2\pp}},
	\end{equation}
	to get that 
	\begin{align}
	\sigma^{\frac{1}{\pp}}\norma{r^{\pp-2}g_k}{}^2&\leq \sigma^{\frac{\pp-1}{\pp}}\norma{g_k}{}^2+\norma{r^{\pp-1}g_k}{}^2
	\leq \sigma \norma{\nabla g_k}{}^2+2\norma{r^{\pp-1}g_k}{}^2,
	\end{align}
	and the proof is completed with $c_\pp=2$.

	\subsection*{Case $\pp \in (1,2)$} Writing down explicitly, we have that 
	\begin{equation}
	\label{eq:proofq12}
\begin{split}
	\norma{r^{\pp-2}g_k}{}^2&=\iint r^{2(\pp-2)}|g_k|^2r\dd r\dd\theta= \frac{1}{2(\pp-2)+2}\iint \dr (r^{2(\pp-2)+2})|g_k|^2 \dd r\dd\theta\\
	&=- \frac{1}{\pp-1}\iint r^{2(\pp-2)+1}g_k\dr g_k r\dd r\dd\theta.
\end{split}
	\end{equation}
	Now let $a>0$ to be chosen later and rewrite the previous integral, without constants in front, as follows 
	\begin{equation}
	\label{bd:proofq12}
	\begin{split}
	&\left|\iint (\dr g_k) g_k r^{2(\pp-2)+1} r\dd r\dd\theta\right|\\
	&\qquad\leq\iint |\dr g_k| \frac{|g_k|^a}{r^a}r^{2(\pp-1)+1+a} |g_k|^{1-a}r\dd r\dd\theta\\
	&\qquad\leq \norma{\dr g_k}{} \left(\iint \frac{|g_k|^{ap}}{r^{ap}}r\dd r\dd\theta\right)^{\frac{1}{p}}\left(\iint r^{(2(\pp-2)+1+a)p'}|g_k|^{(1-a)p'}r\dd r\dd\theta\right)^{\frac{1}{p'}},
	\end{split}
	\end{equation}
	where we have used H\"older's inequality, in particular we need $ \frac{1}{p}+\frac{1}{p'}=\frac{1}{2}$.
	Now we impose that 
	\begin{align}
	ap=2,\qquad (2(\pp-2)+1+a)p'=2(\pp-1).
	\end{align}
	Solving the previous linear system, combined with the H\"older constrains, gives that
	\begin{equation*}
	a=\frac{2-\pp}{\pp}, \qquad p=\frac{2\pp}{2-\pp},\qquad p'=\frac{\pp}{\pp-1}.
	\end{equation*}
	Then by \eqref{eq:proofq12}, \eqref{bd:proofq12} and the choice of the exponents, we have that
	\begin{equation}
	\label{bdn:r(alpha-2)}
	(\pp-1)\norma{r^{\pp-2}g_k}{}^2\leq \norma{\dr g_k}{}\left\|\frac{g_k}{r}\right\|^{\frac{2-\pp}{\pp}}\norma{r^{\pp-1}g_k}{}^{\frac{2(\pp-1)}{\pp}}.
	\end{equation}
	Then multiply \eqref{bdn:r(alpha-2)} by $k^{\frac{2-\pp}{\pp}}$, to get that 
	\begin{align}
	k^{\frac{2-\pp}{\pp}}(\pp-1)\norma{r^{\pp-2}g_k}{}^2&\leq \norma{\dr g_k}{}\left\|\frac{\dtheta g_k}{r}\right\|^{\frac{2-\pp}{\pp}}\norma{r^{\pp-1}g_k}{}^{\frac{2(\pp-1)}{\pp}}\leq \norma{\nabla g_k}{}^{\frac{2}{\pp}}\norma{r^{\pp-1} g_k }{}^{\frac{2(\pp-1)}{\pp}},
	\end{align}
	where in the last line it is crucial that $\pp \in (1,2)$.	
	Then since $1/\pp+(\pp-1)/\pp=1$, we apply Young's inequality, with a free parameter $\lambda$, to get that 
	\begin{align}
	k^{\frac{2-\pp}{\pp}}(\pp-1)\norma{r^{\pp-2}g_k}{}^2\leq \frac{1}{\pp}\lambda^\pp\norma{\nabla g_k}{}^2+\frac{\pp-1}{\pp}\lambda^{-\frac{\pp}{\pp-1}}\norma{r^{\pp-1}g_k}{}^2,
	\end{align}
	or equivalently 
	\begin{align}
	k^{\frac{2-\pp}{\pp}}\pp\lambda^{\frac{\pp}{\pp-1}}\norma{r^{\pp-2}g_k}{}^2\leq \frac{1}{\pp-1}\lambda^{\frac{\pp^2}{\pp-1}}\norma{\nabla g_k}{}^2+\norma{r^{\pp-1}g_k}{}^2.
	\end{align}
	Then let $\sigma=(\pp-1)^{-1}\lambda^{\pp^2/(\pp-1)}$ and we get that 
	\begin{equation}
	k^{\frac{2-\pp}{\pp}}\pp(\pp-1)^{\frac{1}{\pp}}\sigma^{\frac{1}{\pp}}\norma{r^{\pp-2}g_k}{}^2\leq \sigma \norma{\nabla g_k}{}^2+\norma{r^{\pp-1}g_k}{}^2.
	\end{equation}
	Since $k\geq 1$, this concludes the proof of Lemma \ref{lemma:ralpha-2}.
\end{proof}

\subsection*{Acknowledgements}
The authors would like to thank Theodore Drivas and Christian Seis for helpful discussions, and Lorenzo Tamellini for helping producing the
simulation in Figure \ref{fig:coffee3}. This work was in part carried out at Imperial College London, which we thank for the support and hospitality.


\begin{bibdiv}
\begin{biblist}

\bib{BajerEtAl01}{article}{
      author={Bajer, Konrad},
      author={Bassom, Andrew~P},
      author={Gilbert, Andrew~D},
       title={Accelerated diffusion in the centre of a vortex},
        date={2001},
     journal={Journal of Fluid Mechanics},
      volume={437},
       pages={395\ndash 411},
}

\bib{BG98}{article}{
      author={Bassom, Andrew~P.},
      author={Gilbert, Andrew~D.},
       title={The spiral wind-up of vorticity in an inviscid planar vortex},
        date={1998},
     journal={J. Fluid Mech.},
      volume={371},
       pages={109\ndash 140},
         url={http://dx.doi.org/10.1017/S0022112098001955},
}

\bib{BW13}{article}{
      author={Beck, Margaret},
      author={Wayne, C.~Eugene},
       title={Metastability and rapid convergence to quasi-stationary bar
  states for the two-dimensional {N}avier-{S}tokes equations},
        date={2013},
     journal={Proc. Roy. Soc. Edinburgh Sect. A},
      volume={143},
      number={5},
       pages={905\ndash 927},
         url={http://dx.doi.org/10.1017/S0308210511001478},
}

\bib{BGM15I}{article}{
      author={{Bedrossian}, J.},
      author={{Germain}, P.},
      author={{Masmoudi}, N.},
       title={{Dynamics near the subcritical transition of the 3D Couette flow
  I: Below threshold case}},
        date={2015-06},
     journal={ArXiv e-prints},
      eprint={1506.03720},
}

\bib{BGM15II}{article}{
      author={{Bedrossian}, J.},
      author={{Germain}, P.},
      author={{Masmoudi}, N.},
       title={{Dynamics near the subcritical transition of the 3D Couette flow
  II: Above threshold case}},
        date={2015-06},
     journal={ArXiv e-prints},
      eprint={1506.03721},
}

\bib{BCZ15}{article}{
      author={Bedrossian, Jacob},
      author={Coti~Zelati, Michele},
       title={Enhanced dissipation, hypoellipticity, and anomalous small noise
  inviscid limits in shear flows},
        date={2017},
     journal={Arch. Ration. Mech. Anal.},
      volume={224},
      number={3},
       pages={1161\ndash 1204},
}

\bib{BCZGH15}{article}{
      author={Bedrossian, Jacob},
      author={Coti~Zelati, Michele},
      author={Glatt-Holtz, Nathan},
       title={Invariant {M}easures for {P}assive {S}calars in the {S}mall
  {N}oise {I}nviscid {L}imit},
        date={2016},
     journal={Comm. Math. Phys.},
      volume={348},
      number={1},
       pages={101\ndash 127},
}

\bib{BGM15III}{article}{
      author={Bedrossian, Jacob},
      author={Germain, Pierre},
      author={Masmoudi, Nader},
       title={On the stability threshold for the 3{D} {C}ouette flow in
  {S}obolev regularity},
        date={2017},
     journal={Ann. of Math. (2)},
      volume={185},
      number={2},
       pages={541\ndash 608},
}

\bib{BMV14}{article}{
      author={Bedrossian, Jacob},
      author={Masmoudi, Nader},
      author={Vicol, Vlad},
       title={Enhanced dissipation and inviscid damping in the inviscid limit
  of the {N}avier-{S}tokes equations near the two dimensional {C}ouette flow},
        date={2016},
     journal={Arch. Ration. Mech. Anal.},
      volume={219},
      number={3},
       pages={1087\ndash 1159},
         url={http://dx.doi.org/10.1007/s00205-015-0917-3},
}

\bib{BVW16}{article}{
      author={Bedrossian, Jacob},
      author={Vicol, Vlad},
      author={Wang, Fei},
       title={The {S}obolev stability threshold for 2{D} shear flows near
  {C}ouette},
        date={2018},
     journal={J. Nonlinear Sci.},
      volume={28},
      number={6},
       pages={2051\ndash 2075},
}

\bib{CKRZ08}{article}{
      author={Constantin, P.},
      author={Kiselev, A.},
      author={Ryzhik, L.},
      author={Zlatos, A.},
       title={Diffusion and mixing in fluid flow},
        date={2008},
     journal={Ann. of Math. (2)},
      volume={168},
      number={2},
       pages={643\ndash 674},
         url={http://dx.doi.org/10.4007/annals.2008.168.643},
}

\bib{CZDE18}{article}{
      author={{Coti Zelati}, M.},
      author={{Delgadino}, M.G.},
      author={{Elgindi}, T.M.},
       title={{On the relation between enhanced dissipation time-scales and
  mixing rates}},
        date={2018-06},
     journal={ArXiv e-prints},
      eprint={1806.03258},
}

\bib{CZEW19}{article}{
      author={{Coti Zelati}, Michele},
      author={{Elgindi}, Tarek~M.},
      author={{Widmayer}, Klaus},
       title={{Enhanced dissipation in the Navier-Stokes equations near the
  Poiseuille flow}},
        date={2019-01},
     journal={arXiv e-prints},
       pages={arXiv:1901.01571},
      eprint={1901.01571},
}

\bib{CLS17}{article}{
      author={Crippa, Gianluca},
      author={Luc\`a, Renato},
      author={Schulze, Christian},
       title={Polynomial mixing under a certain stationary {E}uler flow},
        date={2019},
     journal={Phys. D},
      volume={394},
       pages={44\ndash 55},
}

\bib{Deng2013}{article}{
      author={Deng, Wen},
       title={Resolvent estimates for a two-dimensional non-self-adjoint
  operator},
        date={2013},
     journal={Commun. Pure Appl. Anal.},
      volume={12},
      number={1},
       pages={547\ndash 596},
}

\bib{DesvillettesVillani01}{article}{
      author={Desvillettes, L.},
      author={Villani, C.},
       title={On the trend to global equilibrium in spatially inhomogeneous
  entropy-dissipating systems: the linear {F}okker-{P}lanck equation},
        date={2001},
     journal={Comm. Pure Appl. Math.},
      volume={54},
      number={1},
       pages={1\ndash 42},
}

\bib{DMS15}{article}{
      author={Dolbeault, Jean},
      author={Mouhot, Cl{\'e}ment},
      author={Schmeiser, Christian},
       title={Hypocoercivity for linear kinetic equations conserving mass},
        date={2015},
     journal={Trans. Amer. Math. Soc.},
      volume={367},
      number={6},
       pages={3807\ndash 3828},
}

\bib{DubrulleNazarenko94}{article}{
      author={Dubrulle, B},
      author={Nazarenko, S},
       title={On scaling laws for the transition to turbulence in uniform-shear
  flows},
        date={1994},
     journal={Euro. Phys. Lett.},
      volume={27},
      number={2},
       pages={129},
}

\bib{GNN09}{article}{
      author={Gallagher, Isabelle},
      author={Gallay, Thierry},
      author={Nier, Francis},
       title={Spectral asymptotics for large skew-symmetric perturbations of
  the harmonic oscillator},
        date={2009},
     journal={Int. Math. Res. Not. IMRN},
      number={12},
       pages={2147\ndash 2199},
}

\bib{Gallay2017}{article}{
      author={Gallay, Thierry},
       title={Enhanced dissipation and axisymmetrization of two-dimensional
  viscous vortices},
        date={2018},
     journal={Arch. Ration. Mech. Anal.},
      volume={230},
      number={3},
       pages={939\ndash 975},
}

\bib{GGS10}{book}{
      author={Giga, Mi-Ho},
      author={Giga, Yoshikazu},
      author={Saal, J\"{u}rgen},
       title={Nonlinear partial differential equations},
      series={Progress in Nonlinear Differential Equations and their
  Applications},
   publisher={Birkh\"{a}user Boston, Inc., Boston, MA},
        date={2010},
      volume={79},
}

\bib{GNRS18}{article}{
      author={{Grenier}, E.},
      author={{Nguyen}, T.~T.},
      author={{Rousset}, F.},
      author={{Soffer}, A.},
       title={{Linear inviscid damping and enhanced viscous dissipation of
  shear flows by using the conjugate operator method}},
        date={2018-04},
     journal={ArXiv e-prints},
      eprint={1804.08291},
}

\bib{Herau2007}{article}{
      author={H{\'e}rau, Fr{\'e}d{\'e}ric},
       title={Short and long time behavior of the {F}okker-{P}lanck equation in
  a confining potential and applications},
        date={2007},
     journal={J. Funct. Anal.},
      volume={244},
      number={1},
       pages={95\ndash 118},
}

\bib{HerauNier2004}{article}{
      author={H\'{e}rau, Fr\'{e}d\'{e}ric},
      author={Nier, Francis},
       title={Isotropic hypoellipticity and trend to equilibrium for the
  {F}okker-{P}lanck equation with a high-degree potential},
        date={2004},
     journal={Arch. Ration. Mech. Anal.},
      volume={171},
      number={2},
       pages={151\ndash 218},
}

\bib{IMM17}{article}{
      author={{Ibrahim}, S.},
      author={{Maekawa}, Y.},
      author={{Masmoudi}, N.},
       title={{On pseudospectral bound for non-selfadjoint operators and its
  application to stability of Kolmogorov flows}},
        date={2017-10},
     journal={ArXiv e-prints},
      eprint={1710.05132},
}

\bib{LatiniBernoff01}{article}{
      author={Latini, M.},
      author={Bernoff, A.J.},
       title={Transient anomalous diffusion in {Poiseuille} flow},
        date={2001},
     journal={Journal of Fluid Mechanics},
      volume={441},
       pages={399\ndash 411},
}

\bib{LiWeiZhang2017}{article}{
      author={{Li}, T.},
      author={{Wei}, D.},
      author={{Zhang}, Z.},
       title={{Pseudospectral and spectral bounds for the Oseen vortices
  operator}},
        date={2017-01},
     journal={ArXiv e-prints},
      eprint={1701.06269},
}

\bib{LWZ18}{article}{
      author={{Li}, Te},
      author={{Wei}, Dongyi},
      author={{Zhang}, Zhifei},
       title={{Pseudospectral bound and transition threshold for the 3D
  Kolmogorov flow}},
        date={2018-01},
     journal={arXiv e-prints},
      eprint={1801.05645},
}

\bib{MD18}{article}{
      author={Miles, Christopher~J.},
      author={Doering, Charles~R.},
       title={Diffusion-limited mixing by incompressible flows},
        date={2018},
     journal={Nonlinearity},
      volume={31},
      number={5},
       pages={2346\ndash 2350},
}

\bib{RhinesYoung83}{article}{
      author={Rhines, P.B.},
      author={Young, W.R.},
       title={How rapidly is a passive scalar mixed within closed
  streamlines?},
        date={1983},
     journal={Journal of Fluid Mechanics},
      volume={133},
       pages={133\ndash 145},
}

\bib{Villani09}{article}{
      author={Villani, C{\'e}dric},
       title={Hypocoercivity},
        date={2009},
     journal={Mem. Amer. Math. Soc.},
      volume={202},
      number={950},
       pages={iv+141},
}

\bib{WZ18}{article}{
      author={{Wei}, D.},
      author={{Zhang}, Z.},
       title={{Transition threshold for the 3D Couette flow in Sobolev space}},
        date={2018-03},
     journal={ArXiv e-prints},
      eprint={1803.01359},
}

\bib{WZZkolmo17}{article}{
      author={{Wei}, D.},
      author={{Zhang}, Z.},
      author={{Zhao}, W.},
       title={{Linear inviscid damping and enhanced dissipation for the
  Kolmogorov flow}},
        date={2017-11},
     journal={ArXiv e-prints},
      eprint={1711.01822},
}

\bib{WEI18}{article}{
      author={{Wei}, Dongyi},
       title={{Diffusion and mixing in fluid flow via the resolvent estimate}},
        date={2018-11},
     journal={arXiv e-prints},
      eprint={1811.11904},
}

\bib{WZ19}{article}{
      author={Wei, Dongyi},
      author={Zhang, Zhifei},
       title={Enhanced dissipation for the {K}olmogorov flow via the
  hypocoercivity method},
        date={2019},
     journal={Sci. China Math.},
      volume={62},
      number={6},
       pages={1219\ndash 1232},
}

\end{biblist}
\end{bibdiv}


\end{document}